\documentclass[letterpaper]{article}

\usepackage{charter}
\usepackage[utf8]{inputenc}
\usepackage[T1]{fontenc}
\usepackage[dvipsnames]{xcolor}
\usepackage[bookmarksnumbered]{hyperref}
\usepackage{enumitem}
\usepackage{amssymb}
\usepackage{amsmath}
\usepackage{amsthm}
\usepackage{stmaryrd}
\usepackage{algorithm}
\usepackage{algorithmic}
\usepackage{mathtools}
\usepackage{tikz}
\usepackage[numbers]{natbib}
\usepackage{booktabs}
\usepackage[font=small]{caption}
\usepackage[bottom]{footmisc}
\usepackage{parskip}
\usepackage{pifont}

\renewcommand{\leq}{\leqslant}
\renewcommand{\geq}{\geqslant}
\renewcommand{\epsilon}{\varepsilon}
\DeclarePairedDelimiter\floor{\lfloor}{\rfloor}

\DeclareMathOperator*{\argmin}{\mathrm{arg\,min}}

\hypersetup{
    colorlinks=true,
    linkcolor={red!60!black},
    citecolor={green!40!black}
}

\newtheorem{theorem}{Theorem}[section]
\newtheorem{corollary}[theorem]{Corollary}
\newtheorem{fact}[theorem]{Fact}
\newtheorem{lemma}[theorem]{Lemma}

\newtheorem{remark}[theorem]{Remark}

\newtheorem{assumption}[theorem]{Assumption}

\begin{document}

\begin{center}
 {\Large Revisiting the Approximate Carath\'eodory Problem\\ via the Frank-Wolfe Algorithm}
\end{center}

\vspace{7mm}

\noindent\textbf{Cyrille W.\ Combettes}\textsuperscript{ 1 3}\hfill\href{mailto:cyrille@gatech.edu}{\ttfamily cyrille@gatech.edu}\\
\\
\textbf{Sebastian Pokutta}\textsuperscript{ 2 3}\hfill\href{mailto:pokutta@zib.de}{\ttfamily pokutta@zib.de}\\
\\
{\small\textsuperscript{1 }\emph{School of Industrial and Systems Engineering, Georgia Institute of Technology, USA}}\\
{\small\textsuperscript{2 }\emph{Institute of Mathematics, Technische Universit\"at Berlin, Germany}}\\
{\small\textsuperscript{3 }\emph{Department for AI in Society, Science, and Technology, Zuse Institute Berlin, Germany}}

\vspace{7mm}

\begin{center}
\begin{minipage}{0.9\textwidth}
\begin{center}
 \textbf{Abstract}
\end{center}
\vspace{3mm}
  {\small The approximate Carath\'eodory theorem states that given a compact convex set $\mathcal{C}\subset\mathbb{R}^n$ and $p\in\left[2,+\infty\right[$, each point $x^*\in\mathcal{C}$ can be approximated to $\epsilon$-accuracy in the $\ell_p$-norm as the convex combination of $\mathcal{O}(pD_p^2/\epsilon^2)$ vertices of $\mathcal{C}$, where $D_p$ is the diameter of $\mathcal{C}$ in the $\ell_p$-norm. A solution satisfying these properties can be built using probabilistic arguments or by applying mirror descent to the dual problem. We revisit the approximate Carath\'eodory problem by solving the primal problem via the Frank-Wolfe algorithm, providing a simplified analysis and leading to an efficient practical method. Furthermore, improved cardinality bounds are derived naturally using existing convergence rates of the Frank-Wolfe algorithm in different scenarios, when $x^*$ is in the interior of $\mathcal{C}$, when $x^*$ is the convex combination of a subset of vertices with small diameter, or when $\mathcal{C}$ is uniformly convex. We also propose cardinality bounds when $p\in\left[1,2\right[\cup\{+\infty\}$ via a nonsmooth variant of the algorithm. Lastly, we address the problem of finding sparse approximate projections onto $\mathcal{C}$ in the $\ell_p$-norm, $p\in\left[1,+\infty\right]$.}
\end{minipage}
\end{center}

\vspace{2mm}

\section{Introduction}

Let $\mathcal{C}\subset\mathbb{R}^n$ be a compact convex set and $x^*\in\mathcal{C}$. Suppose that we are interested in expressing $x^*$ as the convex combination of as few vertices of $\mathcal{C}$ as possible. Motivations for this may lie in, e.g., memory space, computation time, or model interpretability. Then Carath\'eodory's theorem \cite{cara1907} states that this can be achieved with less than $n+1$ vertices, and this bound is tight. However, in the case where we can afford an $\epsilon$-approximation in the $\ell_p$-norm, where $p\in\left[1,+\infty\right]$, can we reduce it to just $m$ points with $m$ being significantly smaller than $n+1$?

We address the \emph{approximate} Carath\'eodory problem, which aims at finding a point $x\in\mathcal{C}$ that is the convex combination of a small number of vertices and satisfying $\|x-x^*\|_p\leq\epsilon$. Let the cardinality of $x$, with respect to a given convex decomposition, be the number of vertices in the decomposition. When $p\in\left[2,+\infty\right[$, the approximate Carath\'eodory theorem states that there exists a solution with cardinality $\mathcal{O}(pD_p^2/\epsilon^2)$, where $D_p$ is the diameter of $\mathcal{C}$ in the $\ell_p$-norm \cite{barman15cara}. The bound is independent of the dimension $n$, and it is therefore very significant in high-dimensional spaces as it shows that we can obtain extremely sparse solutions. Applications in game theory (Nash equilibria) and combinatorial optimization (densest $k$-subgraphs) are presented in \cite{barman15cara}.

The approximate Carath\'eodory theorem can be proved using Maurey's lemma \cite{pisier81}. A similar proof is presented in \cite{barman15cara}, which consists in solving the exact Carath\'eodory problem and then reducing the number of vertices by sampling. A lower bound $\Omega((D_p/\epsilon)^{p/(p-1)})$ on the cardinality is also provided. Later on, a new proof using only deterministic arguments was proposed in \cite{mirrokni17cara}, by building the solution via mirror descent \cite{nemirovsky83md}. This is particularly useful in practice since the method in \cite{barman15cara} is expensive, as solving the exact Carath\'eodory problem has complexity polynomial in $n$ even when the vertices are known \cite{maalouf19}\footnote{In \cite[Thm.~3.1]{maalouf19}, the dimension of the ambient space is denoted by $d$.}. Furthermore, it is proved in \cite{mirrokni17cara} that if $x^*$ is in the interior of $\mathcal{C}$, then a solution with cardinality $\mathcal{O}(p(D_p/r_p)^2\ln(1/\epsilon))$ can be found, where $r_p>0$ denotes the radius of the ball centered at $x^*$ and included in $\mathcal{C}$. Finally, they improved the lower bound to $\Omega(pD_p^2/\epsilon^2)$, thus establishing the optimality of the approximate Carath\'eodory theorem in the general setting.

When $p=+\infty$, there exists a solution with cardinality $\mathcal{O}(\ln(n)D_\infty^2/\epsilon^2)$ \cite{barman15cara}. When $p\in\left]1,2\right[$, a cardinality bound $\mathcal{O}((1/p)^{1/(p-1)}(D_p/\epsilon)^{p/(p-1)})$ can be derived from Maurey's lemma; see \cite[Lem.~D]{bourgain89} and \cite{ivanov19}. In the more general setting of uniformly smooth Banach spaces, an approximate Carath\'eodory theorem was recently proposed in \cite{ivanov19}.

\begin{table}[h]
 \caption{Cardinality bounds to achieve $\epsilon$-convergence in the approximate Carath\'eodory problem with respect to the $\ell_p$-norm. All our bounds are obtained via the Frank-Wolfe algorithm or variants.}
 \label{tab:lit}
 \centering{
 \begin{tabular}{llll}
  \toprule
  \textbf{$\ell_p$-norm}&\textbf{Assumption}&\multicolumn{2}{c}{\textbf{Cardinality bound}}\\
  \cmidrule(lr){3-4}
  &&This paper&Related work\\
  \midrule
  $p\in\left[2,+\infty\right[$&--&$\displaystyle\mathcal{O}\!\left(\frac{pD_p^2}{\epsilon^2}\right)$ or $\displaystyle\mathcal{O}\!\left(\frac{p(D_*^2+D_0^2)}{\epsilon^2}\right)$&$\displaystyle\mathcal{O}\!\left(\frac{pD_p^2}{\epsilon^2}\right)$ \cite{barman15cara,mirrokni17cara,ivanov19}\\
  &&\hfill (Corollaries~\ref{cor:fw}--\ref{cor:nep})&\\
  \cmidrule(lr){2-4}
  &$x^*\in\operatorname{int}\mathcal{C}$&$\displaystyle\mathcal{O}\!\left(p\left(\frac{D_p}{r_p}\right)^2\ln\!\left(\frac{1}{\epsilon}\right)\right)$&$\displaystyle\mathcal{O}\!\left(p\left(\frac{D_p}{r_p}\right)^2\ln\!\left(\frac{1}{\epsilon}\right)\right)$\\
  &&\hfill (Corollary~\ref{cor:int})&\hfill \cite{mirrokni17cara}\\
  \cmidrule(lr){2-4}
  &$\mathcal{C}$ is $\alpha_p$-strongly convex&$\displaystyle\mathcal{O}\!\left(\frac{\sqrt{p}D_p+p/\alpha_p}{\epsilon}\right)$&--\\
  &&\hfill (Corollary~\ref{cor:t2})&\\
  \cmidrule(lr){2-4}
  &$\mathcal{C}$ is $(\alpha_p,q_p)$-uniformly&$\displaystyle\mathcal{O}\!\left(\frac{(pD_p^2)^{(q_p-1)/q_p}+p/\alpha_p^{2/q_p}}{\epsilon^{2(q_p-1)/q_p}}\right)$&--\\
  &\hfill convex, $q_p\in\left[2,+\infty\right[$&\hfill (Corollary~\ref{cor:uni2})&\\
  \midrule
  $p\in\left]1,2\right[$&--&$\displaystyle\mathcal{O}\!\left(\frac{n^{(2-p)/p}D_2^2}{\epsilon^2}\right)$&$\displaystyle\mathcal{O}\!\left(\frac{D_p^{p/(p-1)}}{p^{1/(p-1)}\epsilon^{p/(p-1)}}\right)$\\
  &&\hfill (Corollaries~\ref{cor:hcgs} and~\ref{cor:fw:fbeta})&\hfill\cite{bourgain89,ivanov19}\\
  \midrule
  $p=1$&--&$\displaystyle\mathcal{O}\!\left(\frac{nD_2^2}{\epsilon^2}\right)$&--\\
  &&\hfill (Corollaries~\ref{cor:hcgs} and~\ref{cor:fw:fbeta})&\\
  \midrule
  $p=+\infty$&--&$\displaystyle\mathcal{O}\!\left(\frac{D_2^2}{\epsilon^2}\right)$&$\displaystyle\mathcal{O}\!\left(\frac{\ln(n)D_\infty^2}{\epsilon^2}\right)$ \cite{barman15cara}\\
  &&\hfill (Corollaries~\ref{cor:hcgsinf} and~\ref{cor:fw:fbetainf})&\\
  \bottomrule
 \end{tabular}
 }
\end{table}

\paragraph{Contributions.} We address the approximate Carath\'eodory problem in the $\ell_p$-norm via the Frank-Wolfe algorithm (FW). We cover the whole range $p\in\left[1,+\infty\right]$, with a slight modification of FW when $p\in\left[1,2\right[\cup\{+\infty\}$. When $p\in\left[2,+\infty\right[$, we recover the cardinality bound $\mathcal{O}(pD_p^2/\epsilon^2)$ by addressing the primal problem directly. This is in contrast with the approach in \cite{mirrokni17cara}, which consists of formulating the dual problem and solving it via mirror descent; although it is pointed out that this selects the exact same set of vertices as if FW was applied to the primal problem \cite{bach15dualfw}, our analysis is much simpler. Moreover, the method in \cite{mirrokni17cara} for the case $x^*\in\operatorname{int}\mathcal{C}$ requires restarting mirror descent and knowledge of the radius $r_p$, which may not be available. We show that a direct application of FW generates the desired solution, i.e., that FW is adaptive to the properties of the problem. Our approach further provides improved cardinality bounds when $x^*$ is the convex combination of a subset of vertices with small diameter or when $\mathcal{C}$ is uniformly convex. When $p\in\left[1,2\right[$, we build a solution with cardinality $\mathcal{O}(n^{(2-p)/p}D_2^2/\epsilon^2)$ via a nonsmooth variant of FW. This improves the dependence to $\epsilon$ in the previous known bound for $p\in\left]1,2\right[$ but involves a term (at most linear) in the dimension $n$. The nonsmooth FW variant also finds a solution with cardinality $\mathcal{O}(D_2^2/\epsilon^2)$ when $p=+\infty$, which is dimension-free compared to the result in \cite{barman15cara}. Finally, we address the problem of finding sparse approximate projections in the $\ell_p$-norm.

\paragraph{Outline.} The bulk of the paper considers the case $p\in\left[2,+\infty\right[$. We introduce notation and definitions in Section~\ref{sec:prelim}. In Section~\ref{sec:fw}, we show that the Frank-Wolfe algorithm is an intuitive method to solve the approximate Carath\'eodory problem, and we review its convergence analyses in Section~\ref{sec:cv}. In Section~\ref{sec:cara}, we prove that it solves the approximate Carath\'eodory theorem and that it also generates improved cardinality bounds in different scenarios. In Section~\ref{sec:12}, we analyze the case $p\in\left[1,2\right[\cup\{+\infty\}$ via a nonsmooth variant of the Frank-Wolfe algorithm. In Section~\ref{sec:proj}, we address the problem of finding sparse approximate projections. We present computational experiments in Section~\ref{sec:exp}. We briefly mention in Section~\ref{sec:lower} a correction to the lower bound $\Omega(1/\epsilon^2)$ presented in \cite[Sec.~5.1]{mirrokni17cara}.

\section{Notation and definitions}
\label{sec:prelim}

We consider the Euclidean space $(\mathbb{R}^n,\langle\cdot,\cdot\rangle)$ and an arbitrary norm $\|\cdot\|$. The dual norm of $\|\cdot\|$ is $\|\cdot\|_*\colon y\in\mathbb{R}^n\mapsto\sup_{\|x\|\leq1}\langle x,y\rangle$. For all $x\in\mathbb{R}^n$ and $i\in\llbracket1,n\rrbracket$, let $[x]_i$ be the $i$-th entry of $x$. Given $p\geq1$, the $\ell_p$-norm is $\|\cdot\|_p\colon x\in\mathbb{R}^n\mapsto\left(\sum_{i=1}^n|[x]_i|^p\right)^{1/p}$. For any closed convex set $\mathcal{C}\subset\mathbb{R}^n$, the projection operator onto $\mathcal{C}$ and the distance function to $\mathcal{C}$ in the $\ell_p$-norm are denoted by $\operatorname{proj}_p(\cdot,\mathcal{C})$ and $\operatorname{dist}_p(\cdot,\mathcal{C})$ respectively. For any two sets $\mathcal{A}$ and $\mathcal{B}$ of $\mathbb{R}^n$, $\mathcal{A}$ is included in $\mathcal{B}$, and we write $\mathcal{A}\subset\mathcal{B}$, if for all $x\in\mathcal{A}$, it holds $x\in\mathcal{B}$. The relative interior of a set $\mathcal{C}\subset\mathbb{R}^n$ is denoted by $\operatorname{int}\mathcal{C}$. It is independent of the norm. Given a compact convex set $\mathcal{C}\subset\mathbb{R}^n$, the cardinality of a point $x\in\mathcal{C}$, with respect to a given convex decomposition onto the vertices of $\mathcal{C}$, is the number of vertices (with positive weights) in the decomposition. Informally, we say that $x$ is sparse if it has low cardinality.

A set $\mathcal{C}\subset\mathbb{R}^n$ is $(\alpha,q)$-uniformly convex with respect to $\|\cdot\|$ if $\alpha,q>0$ and for all $x,y\in\mathcal{C}$, $\gamma\in\left[0,1\right]$, and $z\in\mathbb{R}^n$ with $\|z\|=1$,
\begin{align*}
(1-\gamma)x+\gamma y+(1-\gamma)\gamma\alpha\|x-y\|^qz\in\mathcal{C}.
\end{align*}
If $q=2$, we say that $\mathcal{C}$ is $\alpha$-strongly convex with respect to $\|\cdot\|$. Examples of such sets are discussed in, e.g., \cite{kerdreux21}.

Let $\mathcal{C}\subset\mathbb{R}^n$ be a convex set and $f\colon\mathbb{R}^n\to\mathbb{R}$ be a function. Then:
\begin{enumerate}[label=(\roman*)]
 \item $f$ is $G$-Lipschitz-continuous on $\mathcal{C}$ with respect to $\|\cdot\|$ if $G>0$ and for all $x,y\in\mathcal{C}$,
 \begin{align*}
  |f(y)-f(x)|
  \leq G\|y-x\|.
 \end{align*}
 \item $f$ is $L$-smooth on $\mathcal{C}$ with respect to $\|\cdot\|$ if $f$ is differentiable on $\mathcal{C}$, $L>0$, and for all $x,y\in\mathcal{C}$,
\begin{align*}
 f(y)
 \leq f(x)+\langle y-x,\nabla f(x)\rangle+\frac{L}{2}\|y-x\|^2.
\end{align*}
 \item\label{def:sc} $f$ is $S$-strongly convex on $\mathcal{C}$ with respect to $\|\cdot\|$ if $f$ is differentiable on $\mathcal{C}$, $S>0$, and for all $x,y\in\mathcal{C}$,
\begin{align*}
 f(y)
 \geq f(x)+\langle y-x,\nabla f(x)\rangle+\frac{S}{2}\|y-x\|^2.
\end{align*}
\item\label{def:sh} $f$ is $\sigma$-sharp on $\mathcal{C}$ with respect to $\|\cdot\|$ if $\mathcal{C}$ is compact, $\sigma>0$, and for all $x\in\mathcal{C}$,
\begin{align*}
 \min_{x^*\in\argmin_\mathcal{C}f}\|x-x^*\|
 \leq\sigma\sqrt{f(x)-\min_\mathcal{C}f}.
\end{align*}
\item\label{def:gd} $f$ is $\mu$-gradient dominated on $\mathcal{C}$ with respect to $\|\cdot\|$ if $f$ is differentiable on $\mathcal{C}$, $\mu>0$, $\argmin_\mathcal{C}f\neq\varnothing$, and for all $x\in\mathcal{C}$,
\begin{align*}
 f(x)-\min_\mathcal{C}f
 \leq\frac{\|\nabla f(x)\|_*^2}{2\mu}.
\end{align*}
\end{enumerate}

Note that if $f$ is gradient dominated on $\mathbb{R}^n$, then it is gradient dominated on any convex set $\mathcal{C}\subset\mathbb{R}^n$. Facts~\ref{fact:scsh}--\ref{fact:shgd} show the connection between definitions~\ref{def:sc}--\ref{def:gd}. Definition~\ref{def:gd} is often referred to as the Polyak-\L ojasiewicz inequality \cite{polyak63pl,lojo63}. It is a special case of the Kurdyka-\L ojasiewicz property, named after \cite{kurdyka98,lojo63}, satisfied by a very large class of functions \cite{bolte07b} and thus very useful for analyzing optimization algorithms \cite{bolte10a,bolte13,bolte14,bolte17}.

\begin{fact}
 \label{fact:scsh}
 Let $\mathcal{C}\subset\mathbb{R}^n$ be a compact convex set and $f\colon\mathbb{R}^n\to\mathbb{R}$ be differentiable on $\mathcal{C}$. If $f$ is $S$-strongly convex on $\mathcal{C}$ with respect to $\|\cdot\|$, then $f$ is $\sqrt{2/S}$-sharp on $\mathcal{C}$ with respect to $\|\cdot\|$.
\end{fact}

\begin{fact}[\cite{xu18}]
 \label{fact:shgd}
 Let $\mathcal{C}\subset\mathbb{R}^n$ be a compact convex set and $f\colon\mathbb{R}^n\to\mathbb{R}$ be differentiable on $\mathcal{C}$. If $f$ is $\sigma$-sharp on $\mathcal{C}$ with respect to $\|\cdot\|$, then $f$ is $1/(2\sigma^2)$-gradient dominated on $\mathcal{C}$ with respect to $\|\cdot\|$.
\end{fact}

\section{Frank-Wolfe and the approximate Carath\'eodory problem}
\label{sec:fw}

Given a compact convex set $\mathcal{C}\subset\mathbb{R}^n$, a point $x^*\in\mathcal{C}$, and an $\ell_p$-norm where $p\in\left[2,+\infty\right[$, the approximate Carath\'eodory problem can be formulated as the problem of finding a sparse approximate solution to
\begin{align}
 \min_{x\in\mathcal{C}}\frac{1}{2}\|x-x^*\|_p^2.\label{pb:cara}
\end{align}
A natural strategy is to start from an arbitrary vertex $x_0\in\mathcal{C}$ and to sequentially pick up new vertices until the iterates have converged to the desired accuracy. Putting a square on the $\ell_p$-norm provides the objective $f\colon x\in\mathbb{R}^n\mapsto(1/2)\|x-x^*\|_p^2$ with several properties favorable to optimization (Lemma~\ref{lem:ellp2}).

\begin{lemma}
 \label{lem:ellp2}
 Let $\mathcal{C}\subset\mathbb{R}^n$ be a compact convex set, $x^*\in\mathcal{C}$, $p\in\left[2,+\infty\right[$, and $f\colon x\in\mathbb{R}^n\mapsto(1/2)\|x-x^*\|_p^2$. Then $f$ is convex, $(p-1)$-smooth and $1$-gradient dominated on $\mathbb{R}^n$, and $\sqrt{2}$-sharp on $\mathcal{C}$, all respect to the $\ell_p$-norm.
\end{lemma}

\begin{proof}
 The convexity and the $\sqrt{2}$-sharpness of $f$ are trivial. Let $h\colon x\in\mathbb{R}^n\mapsto(1/2)\|x\|_p^2$. For all $q\in\left]1,2\right]$, $g\colon y\in\mathbb{R}^n\mapsto(1/2)\|y\|_q^2$ is $(q-1)$-strongly convex with respect to the $\ell_q$-norm \cite[Lem.~17]{shalev07phd}. Let $q=p/(p-1)\in\left]1,2\right]$. Then the dual norm of the $\ell_q$-norm is the $\ell_p$-norm and the conjugate of $g$ is $h$ \cite[Rem.~I.4.1]{ekeland99}. By \cite[Cor.~3.5.11 and Rem.~3.5.3]{zalinescu02}, $h$ is $1/(q-1)$-smooth with respect to the $\ell_p$-norm, i.e., $f$ is $(p-1)$-smooth with respect to the $\ell_p$-norm. Lastly, let $x\in\mathbb{R}^n$. We have
 \begin{align}
  \nabla f(x)
  =\|x-x^*\|_p^{2-p}
  \begin{pmatrix}
  \vdots\\
  \operatorname{sign}([x-x^*]_i)|[x-x^*]_i|^{p-1}\\
  \vdots
  \end{pmatrix}.\label{ellp:grad}
 \end{align}
 Thus,
 \begin{align*}
  \|\nabla f(x)\|_q^2
  &=\|x-x^*\|_p^{2(2-p)}\left(\sum_{i=1}^n|[x-x^*]_i|^{q(p-1)}\right)^{2/q}\\
  &=\|x-x^*\|_p^{2(2-p)}\left(\sum_{i=1}^n|[x-x^*]_i|^p\right)^{2(p-1)/p}\\
  &=\|x-x^*\|_p^2.
 \end{align*}
 Therefore, $f$ is $1$-gradient dominated with respect to the $\ell_p$-norm.
\end{proof}

In fact, the Frank-Wolfe algorithm (FW) \cite{fw56}, a.k.a.\ conditional gradient algorithm \cite{levitin66}, follows exactly this strategy. FW is presented in Algorithm~\ref{fw} for general smooth convex objectives $f$. At each iteration, it selects a vertex $v_t\in\mathcal{C}$ by solving a linear minimization problem over $\mathcal{C}$ (Line~\ref{fw:lmo}) and moves in its direction with a step-size $\gamma_t\in\left[0,1\right]$ (Line~\ref{fw:new}). That is, it builds the new iterate $x_{t+1}$ as a convex combination of the current iterate $x_t$ and the new vertex $v_t$, effectively adding $v_t$ to the convex decomposition of $x_t$:
\begin{align*}
 x_{t+1}=(1-\gamma_t)x_t+\gamma_tv_t\in\mathcal{C}.
\end{align*}

\begin{algorithm}[h]
\caption{Frank-Wolfe (FW)}
\label{fw}
\begin{algorithmic}[1]
\REQUIRE Start point $x_0\in\mathcal{C}$, step-size strategy $(\gamma_t)_{t\in\mathbb{N}}\in\left[0,1\right]^\mathbb{N}$.
\FOR{$t=0$ \TO $T-1$}
\STATE$v_t\leftarrow\displaystyle\argmin_{v\in\mathcal{C}}\,\langle v,\nabla f(x_t)\rangle$\label{fw:lmo}
\STATE$x_{t+1}\leftarrow x_t+\gamma_t(v_t-x_t)$\label{fw:new}
\ENDFOR
\end{algorithmic}
\end{algorithm}

The vertex $v_t$ minimizes the linear approximation of $f$ at $x_t$ over $\mathcal{C}$, and consequently the sequence $(f(x_t))_{t\in\mathbb{N}}$ converges to $\min_\mathcal{C}f$ (Section~\ref{sec:cv}). Thus, given a desired level of accuracy $\epsilon>0$, we can apply FW to problem~\eqref{pb:cara} and count the number of iterations until $\|x_t-x^*\|_p\leq\epsilon$, i.e., until $f(x_t)-f(x^*)\leq\epsilon^2/2$. We can then provide bounds on the cardinality of the solution based on the convergence analyses of FW. In Section~\ref{sec:cv}, we study these in different scenarios. In practice, since we know the value of $f(x^*)=0$, we can observe the primal gap directly and use it at a stopping criterion to actually realize the cardinality bounds.

We can further improve the algorithm by ensuring that the contribution of each selected vertex is maximized. The Fully-Corrective Frank-Wolfe algorithm (FCFW) \cite{holloway74} computes the new iterate $x_{t+1}$ by reoptimizing $f$ over the convex hull $\operatorname{conv}\{x_0,v_0,\ldots,v_t\}$ of selected vertices. Compared to FW, this avoids selecting redundant vertices in the future. In practice, as we will see in Section~\ref{sec:exp}, FCFW generates iterates with much higher sparsity than FW, although each iteration is more expensive to compute. It is presented in Algorithm~\ref{fcfw}, where $\mathcal{S}_t$ denotes the set of vertices in the convex decomposition of $x_t$.

\begin{algorithm}[h]
\caption{Fully-Corrective Frank-Wolfe (FCFW)}
\label{fcfw}
\begin{algorithmic}[1]
\REQUIRE Start point $x_0\in\mathcal{C}$.
\STATE$\mathcal{S}_0\leftarrow\{x_0\}$
\FOR{$t=0$ \TO $T-1$}
\STATE$v_t\leftarrow\argmin\limits_{v\in\mathcal{C}}\,\langle v,\nabla f(x_t)\rangle$
\STATE$\mathcal{S}_{t+1}\leftarrow\mathcal{S}_t\cup\{v_t\}$
\STATE$x_{t+1}\leftarrow\argmin\limits_{x\in\operatorname{conv}\mathcal{S}_{t+1}}f(x)$\label{fcfw:opt}
\ENDFOR
\end{algorithmic}
\end{algorithm}

\section{Convergence rates of the Frank-Wolfe algorithm}
\label{sec:cv}

In this section, we present convergence rates of the Frank-Wolfe algorithm. Throughout, we consider an arbitrary norm $\|\cdot\|$ on $\mathbb{R}^n$.

\begin{assumption}
 \label{aspt}
 Let $\mathcal{C}\subset\mathbb{R}^n$ be a compact convex set with diameter $D$ and $f\colon\mathbb{R}^n\to\mathbb{R}$ be an $L$-smooth convex function on $\mathcal{C}$, where $D$ and $L$ are defined with respect to $\|\cdot\|$.
\end{assumption}

Under Assumption~\ref{aspt}, the Frank-Wolfe algorithm (FW, Algorithm~\ref{fw}) is a first-order method addressing the constrained convex optimization problem
\begin{align}
 \min_{x\in\mathcal{C}}f(x).\label{pb}
\end{align}

\subsection{The general convergence rate}

There are two step-size strategies for which the convergence of FW has been well studied. The strategy first considered historically \cite{fw56,levitin66,demyanov70} is
\begin{align}
 \gamma_t\leftarrow\min\left\{\frac{\langle x_t-v_t,\nabla f(x_t)\rangle}{L\|x_t-v_t\|^2},1\right\}.\label{closed}
\end{align}
It is obtained by minimizing the quadratic upper bound from smoothness:
\begin{align}
 \gamma_t=\argmin_{\gamma\in\left[0,1\right]}f(x_t)+\gamma\langle v_t-x_t,\nabla f(x_t)\rangle+\frac{L}{2}\gamma^2\|v_t-x_t\|^2,\label{closed:opt}
\end{align}
and guarantees progress at each iteration, i.e., we have always $f(x_{t+1})\leq f(x_t)$. However, it requires some knowledge of the smoothness constant $L$ of $f$. To avoid such a requirement, \emph{open loop} strategies have been proposed \cite{dunn78}, basically in the form $\gamma_t\sim1/t$. We will refer to 
\begin{align}
 \gamma_t\leftarrow\frac{2}{t+2}\label{open}
\end{align}
as the \emph{open-loop} strategy, as used in \cite{jaggi13fw}; the strategy~\eqref{closed} is thus referred to as the \emph{closed-loop} strategy. The open-loop strategy does not ensure progress at each iteration but it is very simple to implement and its oblivious decaying allows analyses of FW in different settings, e.g., with stochastic gradients. Lemma~\ref{lem:init} bounds the primal gap at $x_1$.

\begin{lemma}
 \label{lem:init}
 Let Assumption~\ref{aspt} hold and consider FW (Algorithm~\ref{fw}) with the open-loop strategy~\eqref{open} or the closed-loop strategy~\eqref{closed}. Then
 \begin{align*}
  f(x_1)-\min_\mathcal{C}f
  \leq\frac{LD^2}{2}.
 \end{align*}
\end{lemma}

Under Assumption~\ref{aspt}, FW converges at a rate $\mathcal{O}(LD^2/t)$ \cite{levitin66,jaggi13fw}. Note that the proof technique for the closed-loop strategy was already seen in \cite[Sec.~6]{fw56}.

\begin{theorem}[\cite{levitin66,jaggi13fw}]
 \label{th:fw}
 Let Assumption~\ref{aspt} hold and consider FW (Algorithm~\ref{fw}) with:
 \begin{enumerate}[label=(\roman*)]
  \item the closed-loop strategy~\eqref{closed}. Then for all $t\geq1$,
 \begin{align*}
  f(x_t)-\min_\mathcal{C}f
  \leq\frac{4LD^2}{t+2}.
 \end{align*}
  \item the open-loop strategy~\eqref{open}. Then for all $t\geq1$,
 \begin{align*}
  f(x_t)-\min_\mathcal{C}f
  \leq\frac{2LD^2}{t+2}.
 \end{align*}
 \end{enumerate}
\end{theorem}

The convergence rate $\mathcal{O}(LD^2/t)$ is optimal \cite{canon68,jaggi13fw,lan13}, however it is possible to derive faster rates under additional assumptions.

\subsection{Faster convergence rates under additional assumptions}
\label{sec:faster}

Faster convergence rates can be established under additional assumptions on the geometry of $\mathcal{C}$, the properties of $f$, or the location of the set of unconstrained solutions $\argmin_{\mathbb{R}^n}f$ with respect to $\mathcal{C}$. Note that they do not require modifying the algorithmic design of FW, and the step-size strategy is the same closed-loop strategy~\eqref{closed}. This shows that FW is adaptive and naturally leverages the structure of the problem. A summary is presented in Table~\ref{tab:fw}.

\begin{table}[h]
 \caption{Additional assumptions and corresponding convergence rates of FW on problem~\eqref{pb}, where $\mathcal{C}$ is a compact convex set and $f$ is a smooth convex function (Assumption~\ref{aspt}). We denote by $\mathcal{X}=\argmin_{\mathbb{R}^n}f$ the set of unconstrained solutions, possibly empty. The strong convexity assumption can be generalized to that of uniform convexity and also leads to faster rates (Theorems~\ref{th:uni1}--\ref{th:uni2}).}
\label{tab:fw}
\centering{
 \begin{tabular}{ccccl}
  \toprule
  \multicolumn{4}{c}{\bf Additional assumptions}&{\bf Rate}\\
  \cmidrule(lr){1-4}
  $\mathcal{C}$ strongly convex&$f$ gradient dominated&$\mathcal{X}\cap\operatorname{int}\mathcal{C}\neq\varnothing$&$\mathcal{X}\cap\mathcal{C}=\varnothing$\\
  \midrule
  \textcolor{gray!60}{\ding{55}}&\textcolor{gray!60}{\ding{55}}&\textcolor{gray!60}{\ding{55}}&\textcolor{gray!60}{\ding{55}}&$\mathcal{O}(1/t)$\\
  \textcolor{gray!60}{\ding{55}}&\ding{51}&\ding{51}&\textcolor{gray!60}{\ding{55}}&$\mathcal{O}(\exp(-\omega t))$\\
  \ding{51}&\textcolor{gray!60}{\ding{55}}&\textcolor{gray!60}{\ding{55}}&\ding{51}&$\mathcal{O}(\exp(-\omega t))$\\
  \ding{51}&\ding{51}&\textcolor{gray!60}{\ding{55}}&\textcolor{gray!60}{\ding{55}}&$\mathcal{O}(1/t^2)$\\
  \bottomrule
 \end{tabular}
 }
\end{table}

If there exists an unconstrained solution $x^*\in\argmin_{\mathbb{R}^n}f$ in the interior of $\mathcal{C}$ and if $f$ is gradient dominated, then FW converges at a linear rate, as shown in \cite[Sec.~4.2]{garber15} following an argument similar to \cite{guelat86}.  In Theorem~\ref{th:int}, $r$ is the radius of an ball centered at $x^*$ and included in $\mathcal{C}$.

\begin{theorem}[\cite{garber15,guelat86}]
 \label{th:int}
 In addition to Assumption~\ref{aspt}, suppose that $\argmin_{\mathbb{R}^n}f\cap\operatorname{int}\mathcal{C}\neq\varnothing$ and $f$ is $\mu$-gradient dominated on $\mathcal{C}$ with respect to $\|\cdot\|$. Consider FW (Algorithm~\ref{fw}) with the closed-loop strategy~\eqref{closed}. Then for all $t\geq1$,
 \begin{align*}
  f(x_t)-\min_\mathcal{C}f
  \leq\frac{LD^2}{2}\left(1-\frac{\mu}{L}\left(\frac{r}{D}\right)^2\right)^{t-1},
 \end{align*}
 where $r\in\left]0,D/2\right]$.
\end{theorem}

On the other hand, if all unconstrained solutions are outside of $\mathcal{C}$ and if $\mathcal{C}$ is strongly convex, then FW converges again at a linear rate \cite{levitin66}. The distance of $\argmin_{\mathbb{R}^n}f$ to $\mathcal{C}$ is measured via the quantity $c=\min_\mathcal{C}\|\nabla f\|_*>0$.

\begin{theorem}[\cite{levitin66}]
 \label{th:out}
 In addition to Assumption~\ref{aspt}, suppose that $\mathcal{C}$ is $\alpha$-strongly convex with respect to $\|\cdot\|$ and $\argmin_{\mathbb{R}^n}f\cap\mathcal{C}=\varnothing$. Consider FW (Algorithm~\ref{fw}) with the closed-loop strategy~\eqref{closed}. Then for all $t\geq1$,
 \begin{align*}
  f(x_t)-\min_\mathcal{C}f
  \leq\frac{LD^2}{2}\left(1-\min\left\{\frac{1}{2},\frac{\alpha c}{4L}\right\}\right)^{t-1},
 \end{align*}
 where $c=\min_\mathcal{C}\|\nabla f\|_*>0$.
\end{theorem}

Theorems~\ref{th:int}--\ref{th:out} rely on the location of the set of unconstrained solutions $\argmin_{\mathbb{R}^n}f$ with respect to $\mathcal{C}$, and the convergence rates become increasingly slower as this set comes closer to the boundary of $\mathcal{C}$, which can be seen with $r\to0$ and $c\to0$ respectively. However, if $\mathcal{C}$ is strongly convex and $f$ is gradient dominated, then FW enjoys a faster rate independently of the location of $\argmin_{\mathbb{R}^n}f$ \cite{garber15}.

\begin{theorem}[\cite{garber15}]
 \label{th:t2}
 In addition to Assumption~\ref{aspt}, suppose that $\mathcal{C}$ is $\alpha$-strongly convex and $f$ is $\mu$-gradient dominated on $\mathcal{C}$, both with respect to $\|\cdot\|$. Consider FW (Algorithm~\ref{fw}) with the closed-loop strategy~\eqref{closed}. Then for all $t\geq1$,
 \begin{align*}
  f(x_t)-\min_\mathcal{C}f
  \leq\frac{\max\{(9/2)LD^2,144(L/\alpha)^2/\mu\}}{(t+2)^2}.
 \end{align*}
\end{theorem}

The notion of strong convexity for a set can be generalized to that of uniform convexity. Theorems~\ref{th:uni1}--\ref{th:uni2} are slightly adapted from \cite{kerdreux21} using Lemma~\ref{lem:init}.

\begin{theorem}[\cite{kerdreux21}]
 \label{th:uni1}
 In addition to Assumption~\ref{aspt}, suppose that $\mathcal{C}$ is $(\alpha,q)$-uniformly convex with respect to $\|\cdot\|$, where $q>2$, and $\argmin_{\mathbb{R}^n}f\cap\mathcal{C}=\varnothing$. Consider FW (Algorithm~\ref{fw}) with the closed-loop strategy~\eqref{closed}. Then for all $t\geq1$,
 \begin{align*}
  f(x_t)-\min_\mathcal{C}f
  \leq\frac{\max\{(LD^2/2)(1+\beta_1)^{q/(q-2)},4(L/\beta_2)^{q/(q-2)}(4/(\alpha c))^{2/(q-2)}\}}{(t+\beta_1)^{q/(q-2)}},
 \end{align*}
 where $\beta_1=(2-2^{(q-2)/q})/(2^{(q-2)/q}-1)$, $\beta_2=(q-2)/q-(2/q)(2^{(q-2)/q}-1)$, and $c=\min_\mathcal{C}\|\nabla f\|_*>0$.
\end{theorem}

\begin{theorem}[\cite{kerdreux21}]
 \label{th:uni2}
 In addition to Assumption~\ref{aspt}, suppose that $\mathcal{C}$ is $(\alpha,q)$-uniformly convex, where $q\geq2$, and $f$ is $\sigma$-sharp on $\mathcal{C}$, both with respect to $\|\cdot\|$. Consider FW (Algorithm~\ref{fw}) with the closed-loop strategy~\eqref{closed}. Then for all $t\geq1$,
 \begin{align*}
  f(x_t)-\min_\mathcal{C}f
  \leq\frac{\max\{(LD^2/2)(1+\beta_1)^{q/(q-1)},2(L/\beta_2)^{q/(q-1)}(\sigma/\alpha)^{2/(q-1)}\}}{(t+\beta_1)^{q/(q-1)}},
 \end{align*}
 where $\beta_1=(2-2^{(q-1)/q})/(2^{(q-1)/q}-1)$ and $\beta_2=(q-1)/q-(1/q)(2^{(q-1)/q}-1)$.
\end{theorem}

\subsection{Convergence rate with an enhanced oracle}

A desired feature for the approximate Carath\'eodory problem is to have a faster convergence rate for FW when the solutions have a sparse representation. However, the results in Section~\ref{sec:faster} do not provide such a result. Let $\mathcal{V}\subset\mathcal{C}$ be the set of vertices of $\mathcal{C}$. By replacing the linear minimization problem in FW with 
\begin{align}
 \min_{v\in\mathcal{V}}\,f(x_t)+\gamma_t\langle v-x_t,\nabla f(x_t)\rangle+\frac{L}{2}\gamma_t^2\|v-x_t\|_2^2,\label{lmo2}
\end{align}
which quantity appears when applying the smoothness inequality for $f$ between $x_t$ and $x_t+\gamma_t(v-x_t)$, an improvement on the general convergence rate can be obtained \cite{garber21}. Note that problem~\eqref{lmo2} is constrained to $\mathcal{V}$ instead of $\mathcal{C}$, and can be written
\begin{align}
 \min_{v\in\mathcal{V}}\,\langle v,\nabla f(x_t)\rangle+\lambda_t\|v-x_t\|_2^2,\label{lmo3}
\end{align}
where $\lambda_t=L\gamma_t/2$. In many situations, it actually reduces to a linear minimization problem over $\mathcal{C}$, and therefore does not burden the algorithm. For example, if $\mathcal{V}\subset\{0,1\}^n$, then $\|v\|_2^2=\langle v,1\rangle$ for all $v\in\mathcal{V}$ so
\begin{align*}
 \argmin_{v\in\mathcal{V}}\,\langle v,\nabla f(x_t)\rangle+\lambda_t\|v-x_t\|_2^2
 &=\argmin_{v\in\mathcal{V}}\,\langle v,\nabla f(x_t)\rangle+\lambda_t\|v\|_2^2-2\lambda_t\langle v,x_t\rangle\\
 &=\argmin_{v\in\mathcal{V}}\,\langle v,\nabla f(x_t)+\lambda_t (1-2x_t)\rangle,
\end{align*}
or, if $\|u\|_2=\|v\|_2$ for all $u,v\in\mathcal{V}$, then
\begin{align*}
 \argmin_{v\in\mathcal{V}}\,\langle v,\nabla f(x_t)\rangle+\lambda_t\|v-x_t\|_2^2
 &=\argmin_{v\in\mathcal{V}}\,\langle v,\nabla f(x_t)-2\lambda_t x_t\rangle.
\end{align*}
Since problem~\eqref{lmo3} is equivalent to $\min_{v\in\mathcal{V}}\|v-(x_t-\nabla f(x_t)/(2\lambda_t))\|_2^2$, it is called the \emph{nearest extreme point} (NEP) oracle. The algorithm is presented in Algorithm~\ref{nep}, where the smoothness constant $L$ is with respect to the $\ell_2$-norm.

\begin{algorithm}[h]
\caption{Frank-Wolfe with a Nearest Extreme Point oracle (NEP-FW)}
\label{nep}
\begin{algorithmic}[1]
\REQUIRE Start point $x_0\in\mathcal{C}$, smoothness constant $L>0$, step-size strategy $(\gamma_t)_{t\in\mathbb{N}}\in\left[0,1\right]^\mathbb{N}$.\\
\FOR{$t=0$ \TO $T-1$}
\STATE$v_t\leftarrow\displaystyle\argmin_{v\in\mathcal{V}}\,\langle v,\nabla f(x_t)\rangle+\frac{L}{2}\gamma_t\|v-x_t\|_2^2$\label{nep:lmo}
\STATE$x_{t+1}\leftarrow x_t+\gamma_t(v_t-x_t)$
\ENDFOR
\end{algorithmic}
\end{algorithm}

With this modification, NEP-FW improves the bound in the convergence rate of FW \cite{garber21}. Comparing to Theorem~\ref{th:fw}, the improvement is significant when the solutions lie in the convex hull of a subset of vertices with small diameter and the start point $x_0$ is of good quality.

\begin{theorem}[\cite{garber21}]
 \label{th:nep}
 Let Assumption~\ref{aspt} hold with respect to the $\ell_2$-norm and consider NEP-FW (Algorithm~\ref{nep}) with the open-loop strategy~\eqref{open}. Then for all $t\geq1$,
 \begin{align*}
  f(x_t)-\min_\mathcal{C}f
  \leq\frac{2L(D_*^2+D_0^2)}{t+2},
 \end{align*}
 where $D_*=\min_{\mathcal{S}\subset\mathcal{V},\argmin_\mathcal{C}f\subset\operatorname{conv}\mathcal{S}}\operatorname{diam}\mathcal{S}$ and $D_0=\operatorname{diam}\{v\in\mathcal{V}\mid f(v)\leq f(x_0)\}$ are defined with respect to the $\ell_2$-norm.
\end{theorem}

\section{Application to the approximate Carath\'eodory problem}
\label{sec:cara}

As explained in Section~\ref{sec:fw}, we can obtain a solution to the approximate Carath\'eodory problem by applying the Frank-Wolfe algorithm to problem~\eqref{pb:cara} and the cardinality of the solution can be derived from the convergence analyses in Section~\ref{sec:cv}, with Lemma~\ref{lem:ellp2}. We denote by $\mathcal{V}\subset\mathcal{C}$ the set of vertices and by $D_p=\max_{x,y\in\mathcal{C}}\|y-x\|_p$ the diameter of $\mathcal{C}$ in the $\ell_p$-norm. Note that here, by Lemma~\ref{lem:ellp2} and~\eqref{ellp:grad}, the closed-loop strategy~\eqref{closed} reads
\begin{align}
 \gamma_t
 \leftarrow\min\left\{\frac{\|x_t-x^*\|_p^{2-p}\langle x_t-v_t,\operatorname{sign}(x_t-x^*)|x_t-x^*|^{p-1}\rangle}{(p-1)\|x_t-v_t\|_p^2},1\right\},\label{p:closed}
\end{align}
where $\operatorname{sign}(x_t-x^*)|x_t-x^*|^{p-1}=(\operatorname{sign}([x_t-x^*]_i)|[x_t-x^*]_i|^{p-1})_{i\in\llbracket1,n\rrbracket}\in\mathbb{R}^n$.

Corollary~\ref{cor:fw} follows from Theorem~\ref{th:fw} and shows that FW generates a solution with the optimal $\mathcal{O}(pD_p^2/\epsilon^2)$ number of vertices. Therefore, a solution to the approximate Carath\'eodory problem in the $\ell_p$-norm can be obtained via FW.

\begin{corollary}
 \label{cor:fw}
 By running FW (Algorithm~\ref{fw}) on problem~\eqref{pb:cara} with the closed-loop strategy~\eqref{p:closed} or the open-loop strategy~\eqref{open}, we obtain a point $x\in\mathcal{C}$ with cardinality $\mathcal{O}(pD_p^2/\epsilon^2)$ satisfying $\|x-x^*\|_p\leq\epsilon$.
\end{corollary}

Another possibility is to run NEP-FW. Following Theorem~\ref{th:nep}, Corollary~\ref{cor:nep} shows that we can obtain a better cardinality bound when $x^*$ is the convex combination of a subset of vertices with small diameter and $x_0$ is a good start.

\begin{corollary}
 \label{cor:nep}
 By running NEP-FW (Algorithm~\ref{nep}) on problem~\eqref{pb:cara} with the open-loop strategy~\eqref{open}, we obtain a point $x\in\mathcal{C}$ with cardinality $\mathcal{O}(p(D_*^2+D_0^2)/\epsilon^2)$ satisfying $\|x-x^*\|_p\leq\epsilon$, where $D_*=\min_{\mathcal{S}\subset\mathcal{V},x^*\in\operatorname{conv}\mathcal{S}}\operatorname{diam}\mathcal{S}$ and $D_0=\operatorname{diam}\{v\in\mathcal{V}\mid f(v)\leq f(x_0)\}$ are defined with respect to the $\ell_2$-norm.
\end{corollary}

\begin{proof}
 The result follows from Theorem~\ref{th:nep} and Lemma~\ref{lem:ellp2}, because $f$ is also $(p-1)$-smooth with respect to the $\ell_2$-norm since $\|\cdot\|_p\leq\|\cdot\|_2$ for $p\in\left[2,+\infty\right[$.
\end{proof}

It is likely that the point $x^*$ we want to approximate is in the interior of $\mathcal{C}$. Following Theorem~\ref{th:int}, Corollary~\ref{cor:int} improves the cardinality bound in this scenario to a logarithmic dependence on $1/\epsilon$. Note that a similar result is obtained in \cite{mirrokni17cara}, however they assume knowledge of the radius of an ball centered at $x^*$, which may not be available. This is not required in FW as the method naturally adapts to the structure of the problem.

\begin{corollary}
 \label{cor:int}
 Suppose that $x^*\in\operatorname{int}\mathcal{C}$ and let $r_p$ be the radius of an affine ball centered at $x^*$ and included in $\mathcal{C}$, with respect to the $\ell_p$-norm. Then by running FW (Algorithm~\ref{fw}) on problem~\eqref{pb:cara} with the closed-loop strategy~\eqref{p:closed}, we obtain a point $x\in\mathcal{C}$ with cardinality $\mathcal{O}(p(D_p/r_p)^2\ln(1/\epsilon))$ satisfying $\|x-x^*\|_p\leq\epsilon$.
\end{corollary}

Another scenario is when $\mathcal{C}$ has a particular shape. Following Theorem~\ref{th:t2}, Corollary~\ref{cor:t2} shows an improved cardinality bound when $\mathcal{C}$ is strongly convex. It is actually subsumed by Corollary~\ref{cor:uni2}, which follows from Theorem~\ref{th:uni2}. Note that $2(q_p-1)/q_p\in\left[1,2\right[$ for $q_p\in\left[2,+\infty\right[$.

\begin{corollary}
\label{cor:t2}
 Suppose that $\mathcal{C}$ is $\alpha_p$-strongly convex with respect to the $\ell_p$-norm. Then by running FW (Algorithm~\ref{fw}) on problem~\eqref{pb:cara} with the closed-loop strategy~\eqref{p:closed}, we obtain a point $x\in\mathcal{C}$ with cardinality $\mathcal{O}((\sqrt{p}D_p+p/\alpha_p)/\epsilon)$ satisfying $\|x-x^*\|_p\leq\epsilon$.
\end{corollary}

\begin{corollary}
 \label{cor:uni2}
 Suppose that $\mathcal{C}$ is $(\alpha_p,q_p)$-uniformly convex with respect to the $\ell_p$-norm, where $q_p\in\left[2,+\infty\right[$. Then by running FW (Algorithm~\ref{fw}) on problem~\eqref{pb:cara} with the closed-loop strategy~\eqref{p:closed}, we obtain a point $x\in\mathcal{C}$ with cardinality $\mathcal{O}((pD_p^2)^{(q_p-1)/q_p}+p/\alpha_p^{2/q_p})/\epsilon^{2(q_p-1)/q_p})$ satisfying $\|x-x^*\|_p\leq\epsilon$.
\end{corollary}

\begin{proof}
 The result follows from Theorem~\ref{th:uni2} and Lemma~\ref{lem:ellp2}, with $\beta_1\leq\sqrt{2}$ and $1/\beta_2\leq2+\sqrt{2}$.
\end{proof}

\section{The case \texorpdfstring{$p\in\left[1,2\right[\cup\{+\infty\}$}{p in [1,2[ U \{+inf\}}}
\label{sec:12}

In this section, we study the approximate Carath\'eodory problem when $p\in\left[1,2\right[\cup\{+\infty\}$. In this case, the function $x\in\mathbb{R}^n\mapsto(1/2)\|x-x^*\|_p^2$ is no longer smooth so we cannot apply the Frank-Wolfe algorithm directly. The problem is to find a sparse approximate solution to
\begin{align}
 \min_{x\in\mathcal{C}}\|x-x^*\|_p,\label{pb:12}
\end{align}
where the objective $x\in\mathbb{R}^n\mapsto\|x-x^*\|_p$ is convex and $1$-Lipschitz-continuous with respect to the $\ell_p$-norm, by the triangle inequality, but not smooth. Note that, compared to problem~\eqref{pb:cara} for $p\in\left[2,+\infty\right[$, we removed the square so that the objective is Lipschitz-continuous.

Similarly to the case $p\in\left[2,+\infty\right[$, we will present the convergence rate of a variant of the Frank-Wolfe algorithm, then deduce a bound for the approximate Carath\'eodory problem. We consider the general problem
\begin{align}
 \min_{x\in\mathcal{C}}f(x),\label{pb:hcgs}
\end{align}
where $f\colon\mathbb{R}^n\to\mathbb{R}$ is a convex, continuous, but possibly nonsmooth function (Assumption~\ref{aspt2}). We can smoothen $f$ via its Moreau envelope $f_\beta\colon x\in\mathbb{R}^n\mapsto\min_{y\in\mathbb{R}^n}f(y)+(1/(2\beta))\|x-y\|_2^2$ \cite{moreau65}, where $\beta>0$ is the smoothing parameter. The Moreau envelope is a smooth convex function (Lemma~\ref{lem:fbeta}). The proximity operator of $f$ is $\operatorname{prox}_f\colon x\in\mathbb{R}^n\mapsto\argmin_{y\in\mathbb{R}^n}f(y)+(1/2)\|x-y\|_2^2$ \cite{moreau62}.

\begin{assumption}
 \label{aspt2}
 Let $\mathcal{C}\subset\mathbb{R}^n$ be a compact convex set with diameter $D_2$ and $f\colon\mathbb{R}^n\to\mathbb{R}$ be a convex $G_2$-Lipschitz-continuous function, all with respect to the $\ell_2$-norm.
\end{assumption}

\begin{lemma}[{\cite[Prop.~12.15 and Prop.~12.30]{bc17}}]
 \label{lem:fbeta}
 Let $f\colon\mathbb{R}^n\to\mathbb{R}$ be a convex continuous function and $\beta>0$. Then $f_\beta\colon x\in\mathbb{R}^n\mapsto\min_{y\in\mathbb{R}^n}f(y)+(1/(2\beta))\|x-y\|_2^2$ is convex and $1/\beta$-smooth with respect to the $\ell_2$-norm. Its gradient at $x\in\mathbb{R}^n$ is
 \begin{align}
  \nabla f_{\beta}(x)
  =\frac{1}{\beta}(x-\operatorname{prox}_{\beta f}(x)).\label{fbeta:grad}
 \end{align}
\end{lemma}

Our analysis will rely on the $\ell_2$-norm. Lemma~\ref{lem:ellp} states the properties of the objective in the approximate Carath\'eodory problem~\eqref{pb:12} with respect to the $\ell_2$-norm.

\begin{lemma}
 \label{lem:ellp}
 Let $\mathcal{C}\subset\mathbb{R}^n$ be a compact convex set, $x^*\in\mathcal{C}$, $p\in\left[1,2\right[\cup\{+\infty\}$, and $f\colon x\in\mathbb{R}^n\mapsto\|x-x^*\|_p$. Then $f$ is convex and Lipschitz-continuous on $\mathbb{R}^n$ with respect to the $\ell_2$-norm, with constant $n^{1/p-1/2}$ if $p\in\left[1,2\right[$, else $1$ if $p=+\infty$.
\end{lemma}

\begin{proof}
 By the triangle inequality, the function $f$ is $1$-Lipschitz-continuous with respect to the $\ell_p$-norm. We conclude using $\|\cdot\|_p\leq n^{1/p-1/2}\|\cdot\|_2$ for $p\in\left[1,2\right[$ and $\|\cdot\|_\infty\leq\|\cdot\|_2$.
\end{proof}

We will present two methods to find an $\epsilon$-approximate solution to problem~\eqref{pb:12} with the same cardinality guarantees. When $p\in\left[1,2\right[$, we ensure a cardinality bound $\mathcal{O}(n^{(2-p)/p}D_2^2/\epsilon^2)$. Compared to the bound $\mathcal{O}((1/p)^{1/(p-1)}(D_p/\epsilon)^{p/(p-1)})$ obtained from \cite[Lem.~D]{bourgain89} for $p\in\left]1,2\right[$ (see \cite{ivanov19}), it improves the dependence to the accuracy $\epsilon$ but introduces a factor that is linear in the dimension $n$ in the worst case: $n^{(2-p)/p}<n$ for all $p\in\left]1,2\right[$. This is probably due to our approach working with the $\ell_2$-norm. Note however that when $p=1$, the bound is $\mathcal{O}(nD_2^2/\epsilon^2)$, which is not an improvement on the bound $\mathcal{O}(n)$ from the exact Carath\'eodory theorem. When $p=+\infty$, we ensure a cardinality bound $\mathcal{O}(D_2^2/\epsilon^2)$. This is a dimension-free result compared to the bound $\mathcal{O}(\ln(n)D_\infty^2/\epsilon^2)$ from \cite{barman15cara}. 

\subsection{A nonsmooth variant of the Frank-Wolfe algorithm}

The first method to solve problem~\eqref{pb:hcgs} is to use a variant of the Frank-Wolfe algorithm, the Hybrid Conditional Gradient-Smoothing algorithm (HCGS, Algorithm~\ref{hcgs}), developed in \cite{argyriou14} for addressing composite convex problems.  The idea is to design a strategy $(\beta_t)_{t\in\mathbb{N}}$ and to use the gradient $\nabla f_{\beta_t}(x_t)$ as a surrogate in the Frank-Wolfe algorithm. 

\begin{algorithm}[h]
\caption{Hybrid Conditional Gradient-Smoothing (HCGS)}
\label{hcgs}
\begin{algorithmic}[1]
\REQUIRE Start point $x_0\in\mathcal{C}$, smoothing strategy $(\beta_t)_{t\in\mathbb{N}}\in\mathbb{R}_{++}^\mathbb{N}$, step-size strategy $(\gamma_t)_{t\in\mathbb{N}}\in\left[0,1\right]^\mathbb{N}$.\\
\FOR{$t=0$ \TO $T-1$}
\STATE$v_{t}\leftarrow\argmin\limits_{v\in\mathcal{C}}\,\langle v,\nabla f_{\beta_t}(x_t)\rangle$\hfill$\triangleright${ see~\eqref{fbeta:grad}}\label{hcgs:lmo}
\STATE$x_{t+1}\leftarrow x_t+\gamma_t(v_t-x_t)$
\ENDFOR
\end{algorithmic}
\end{algorithm}

Theorem~\ref{th:hcgs} presents the convergence rate of HCGS. It is adapted from \cite{argyriou14} and uses a slightly different smoothing strategy. We present a proof in Appendix~\ref{apx:proofs} for completeness. Corollaries~\ref{cor:hcgs}--\ref{cor:hcgsinf} present the cardinality bounds for the approximate Carath\'eodory problem using Lemma~\ref{lem:ellp}.

\begin{theorem}[\cite{argyriou14}]
 \label{th:hcgs}
 Let Assumption~\ref{aspt2} hold and consider HCGS (Algorithm~\ref{hcgs}) with the open-loop strategy~\eqref{open} and $\beta_t\leftarrow2(D_2/G_2)/\sqrt{t+2}$ for all $t\in\mathbb{N}$. Then for all $t\geq2$,
 \begin{align*}
  f(x_t)-\min_\mathcal{C}f
  \leq\frac{4G_2D_2}{\sqrt{t+1}}.
 \end{align*}
\end{theorem}

\begin{corollary}
 \label{cor:hcgs}
 Let $p\in\left[1,2\right[$. By running HCGS (Algorithm~\ref{hcgs}) on problem~\eqref{pb:12} with the open-loop strategy~\eqref{open} and $\beta_t\leftarrow2(D_2/n^{1/p-1/2})/\sqrt{t+2}$ for all $t\in\mathbb{N}$, we obtain a point $x\in\mathcal{C}$ with cardinality $\mathcal{O}(n^{(2-p)/p}D_2^2/\epsilon^2)$ satisfying $\|x-x^*\|_p\leq\epsilon$.
\end{corollary}

\begin{corollary}
 \label{cor:hcgsinf}
 By running HCGS (Algorithm~\ref{hcgs}) on problem~\eqref{pb:12} with the open-loop strategy~\eqref{open} and $\beta_t\leftarrow2D_2/\sqrt{t+2}$ for all $t\in\mathbb{N}$, we obtain a point $x\in\mathcal{C}$ with cardinality $\mathcal{O}(D_2^2/\epsilon^2)$ satisfying $\|x-x^*\|_\infty\leq\epsilon$.
\end{corollary}

\subsection{Applying Frank-Wolfe to the smoothed objective}

We can view HCGS as applying one iteration of the Frank-Wolfe algorithm to the sequence of problems
\begin{align*}
 \min_{x\in\mathcal{C}}f_{\beta_t}(x).
\end{align*}
However, if the desired level of accuracy $\epsilon>0$ is given, which is probably the case in the approximate Carath\'eodory problem, then we can apply the Frank-Wolfe algorithm to the fixed problem
\begin{align}
 \min_{x\in\mathcal{C}}f_\beta(x),\label{pb:fbeta}
\end{align}
where $\beta\leftarrow\epsilon/G_2^2$, to obtain a solution to the original problem~\eqref{pb:hcgs}. Theorem~\ref{th:fw:fbeta} formalizes this statement, and Corollaries~\ref{cor:fw:fbeta}--\ref{cor:fw:fbetainf} present the cardinality bounds for the approximate Carath\'eodory problem using Lemma~\ref{lem:ellp}.

\begin{theorem}
 \label{th:fw:fbeta}
 Let Assumption~\ref{aspt2} hold. Let $\epsilon>0$ be the desired level of accuracy and set $\beta\leftarrow\epsilon/G_2^2$. Then by running FW (Algorithm~\ref{fw}) on problem~\eqref{pb:fbeta} with the open-loop strategy~\eqref{open}, we obtain a point $x\in\mathcal{C}$ with cardinality $\mathcal{O}(G_2^2D_2^2/\epsilon^2)$ satisfying $f(x)-\min_\mathcal{C}f\leq\epsilon$.
\end{theorem}

\begin{proof}
 By Lemma~\ref{lem:fbeta}, $f_\beta$ is convex and $1/\beta$-smooth with respect to the $\ell_2$-norm. By Theorem~\ref{th:fw}, after $\floor{4G_2^2D_2^2/\epsilon^2}$ iterations, we have a point $x\in\mathcal{C}$ such that
 \begin{align*}
  f_\beta(x)-\min_\mathcal{C}f_\beta
  &\leq\frac{2D_2^2/\beta}{\floor{4G_2^2D_2^2/\epsilon^2}+2}\\
  &\leq\frac{2G_2^2D_2^2/\epsilon}{4G_2^2D_2^2/\epsilon^2}\\
  &=\frac{\epsilon}{2}.
 \end{align*}
 By \cite[Lem.~4.2]{argyriou14}, $f_\beta\leq f\leq f_\beta+\beta G_2^2/2$ so
 \begin{align*}
  f(x)-\min_\mathcal{C}f
  &\leq f_\beta(x)+\frac{\beta G_2^2}{2}-\min_\mathcal{C}f_\beta\\
  &\leq\epsilon.
 \end{align*}
\end{proof}

\begin{corollary}
 \label{cor:fw:fbeta}
 Let $p\in\left[1,2\right[$. Let $\epsilon>0$ be the desired level of accuracy and set $\beta\leftarrow\epsilon/n^{(2-p)/p}$. Then by running FW (Algorithm~\ref{fw}) on problem~\eqref{pb:fbeta} with the open-loop strategy~\eqref{open}, we obtain a point $x\in\mathcal{C}$ with cardinality $\mathcal{O}(n^{(2-p)/p}D_2^2/\epsilon^2)$ satisfying $\|x-x^*\|_p\leq\epsilon$.
\end{corollary}

\begin{corollary}
 \label{cor:fw:fbetainf}
 Let $\epsilon>0$ be the desired level of accuracy and set $\beta\leftarrow\epsilon$. Then by running FW (Algorithm~\ref{fw}) on problem~\eqref{pb:fbeta} with the open-loop strategy~\eqref{open}, we obtain a point $x\in\mathcal{C}$ with cardinality $\mathcal{O}(D_2^2/\epsilon^2)$ satisfying $\|x-x^*\|_\infty\leq\epsilon$.
\end{corollary}

\section{Sparse approximate projections in the \texorpdfstring{$\ell_p$}{lp}-norm}
\label{sec:proj}

The Frank-Wolfe algorithm can also be used to find sparse approximate projections in the $\ell_p$-norm, where $p\in\left[1,+\infty\right]$. When $p\in\left[2,+\infty\right[$, this is problem~\eqref{pb:cara} with $x^*\notin\mathcal{C}$. Then the same cardinality bounds from Corollaries~\ref{cor:fw}--\ref{cor:nep}, \ref{cor:t2}, and~\ref{cor:uni2} hold with 
\begin{align}
 \|x-x^*\|_p^2-\operatorname{dist}_p(x^*,\mathcal{C})^2\leq\epsilon^2.\label{acc}
\end{align}
Furthermore, Corollaries~\ref{cor:out}--\ref{cor:uni1} follow from Theorems~\ref{th:out} and~\ref{th:uni1}, together with Lemma~\ref{lem:ellp2}. Note that to apply Corollary~\ref{cor:nep} here, $D_*$ is defined with respect to $\operatorname{proj}_p(x^*,\mathcal{C})$ instead of $x^*$. When $p\in\left[1,2\right[\cup\{+\infty\}$, then the same cardinality bounds from Corollaries~\ref{cor:hcgs}--\ref{cor:hcgsinf} and~\ref{cor:fw:fbeta}--\ref{cor:fw:fbetainf} hold with 
\begin{align*}
 \|x-x^*\|_p-\operatorname{dist}_p(x^*,\mathcal{C})\leq\epsilon.
\end{align*}

\begin{corollary}
\label{cor:out}
 Let $p\in\left[2,+\infty\right[$ and $x^*\in\mathbb{R}^n\setminus\mathcal{C}$. Suppose that $\mathcal{C}$ is $\alpha_p$-strongly convex with respect to the $\ell_p$-norm. Then by running FW on problem~\eqref{pb:cara} with the closed-loop strategy~\eqref{p:closed}, we obtain a point $x\in\mathcal{C}$ with cardinality $\mathcal{O}(p/(c_p\alpha_p)\cdot\ln(1/\epsilon))$ satisfying
 \begin{align*}
  \|x-x^*\|_p^2-\operatorname{dist}_p(x^*,\mathcal{C})^2\leq\epsilon^2,
 \end{align*}
 where $c_p=\operatorname{dist}_p(x^*,\mathcal{C})>0$.
\end{corollary}

\begin{proof}
 The bound follows from Theorem~\ref{th:out} and Lemma~\ref{lem:ellp2}. Since for $f\colon x\in\mathbb{R}^n\mapsto(1/2)\|x-x^*\|_p^2$, we have $\nabla f(x)=\|x-x^*\|_p^{2-p}\operatorname{sign}(x-x^*)|x-x^*|^{p-1}$ for all $x\in\mathbb{R}^n$, where $\operatorname{sign}(x-x^*)|x-x^*|^{p-1}=(\operatorname{sign}([x-x^*]_i)|[x-x^*]_i|^{p-1})_{i\in\llbracket1,n\rrbracket}\in\mathbb{R}^n$, we obtain
 \begin{align*}
  \|\nabla f(x)\|_{p/(p-1)}
  &=\|x-x^*\|_p^{2-p}\|\operatorname{sign}(x-x^*)|x-x^*|^{p-1}\|_{p/(p-1)}\\
  &=\|x-x^*\|_p^{2-p}\left(\sum_{i=1}^n|[x-x^*]_i|^p\right)^{(p-1)/p}\\
  &=\|x-x^*\|_p^{2-p+p-1}\\
  &=\|x-x^*\|_p.
 \end{align*}
 Thus, $c_p=\min_{x\in\mathcal{C}}\|\nabla f\|_{p/(p-1)}=\operatorname{dist}_p(x^*,\mathcal{C})$.
\end{proof}

\begin{corollary}
\label{cor:uni1}
 Let $p\in\left[2,+\infty\right[$ and $x^*\in\mathbb{R}^n\setminus\mathcal{C}$. Suppose that $\mathcal{C}$ is $(\alpha_p,q_p)$-uniformly convex with respect to the $\ell_p$-norm, where $q_p>2$. Then by running FW on problem~\eqref{pb:cara} with the closed-loop strategy~\eqref{p:closed}, we obtain a point $x\in\mathcal{C}$ with cardinality $\mathcal{O}(((pD_p^2)^{(q_p-2)/q_p}(1+\beta_1)+(p/\beta_2)/(\alpha_pc_p)^{2/q})/\epsilon^{2(q_p-2)/q_p})$ satisfying
 \begin{align*}
  \|x-x^*\|_p^2-\operatorname{dist}_p(x^*,\mathcal{C})^2\leq\epsilon^2,
 \end{align*}
 where $\beta_1=(2-2^{(q_p-2)/q_p})/(2^{(q_p-2)/q_p}-1)$, $\beta_2=(q_p-2)/q_p-(2/q_p)(2^{(q_p-2)/q_p}-1)$, and $c_p=\operatorname{dist}_p(x^*,\mathcal{C})>0$.
\end{corollary}

\begin{proof}
 The proof follows the same arguments as in the proof of Corollary~\ref{cor:out}.
\end{proof}

\begin{remark}
 In~\eqref{acc} and Corollaries~\ref{cor:out}--\ref{cor:uni1}, if $p=2$ then the same results hold with
 \begin{align*}
  \|x-\operatorname{proj}_2(x^*,\mathcal{C})\|_p\leq\epsilon.
 \end{align*}
 Indeed, let $\pi_2=\operatorname{proj}_2(x^*,\mathcal{C})$. Then $\langle x-\pi_2,x^*-\pi_2\rangle\leq0$, so
 \begin{align*}
  \|x-x^*\|_2^2
  &=\|x-\pi_2\|_2^2+\|\pi_2-x^*\|_2^2+2\langle x-\pi_2,\pi_2-x^*\rangle\\
  &\geq\|x-\pi_2\|_2^2+\|\pi_2-x^*\|_2^2.
 \end{align*}
\end{remark}

\section{Computational experiments}
\label{sec:exp}

We compare the cardinality bounds of the solutions obtained by FW (Algorithm~\ref{fw}), FCFW (Algorithm~\ref{fcfw}), NEP-FW (Algorithm~\ref{nep}), and the Away-Step Frank-Wolfe algorithm (AFW) \cite{wolfe70} on the approximate Carath\'eodory problem~\eqref{pb:cara} for $p\in\left[2,+\infty\right[$. While FW moves only towards vertices, AFW is a variant of FW that allows to move also away from vertices, and converges at a linear rate for smooth strongly convex objectives \cite{lacoste15}. Although the objective in~\eqref{pb:cara} is not strongly convex with respect to the $\ell_p$-norm when $p\in\left]2,+\infty\right[$, it is still interesting to investigate its performance. We also compared to the Pairwise Frank-Wolfe algorithm \cite{lacoste15} but it performed very similarly to AFW.

The experiments were run on a laptop under Linux Ubuntu 20.04 with Intel Core i7-10750H. The code is available at \href{https://github.com/cyrillewcombettes/approxcara}{\textcolor{blue}{\ttfamily https://github.com/cyrillewcombettes/approxcara}}.

\subsection{Dense vs.\ sparse target \texorpdfstring{$x^*$}{x*}}

We generated a set $\mathcal{V}\subset\mathbb{R}^{500}$ of $501$ random points and let $\mathcal{C}=\operatorname{conv}\mathcal{V}$. Then, we generated the target point $x^*\in\mathcal{C}$ as a random convex combination of points in $\mathcal{V}$ or as a sparse random convex combination of points in $\mathcal{V}$, i.e., we randomly selected $25$ points from $\mathcal{V}$ and we generated $x^*$ as a convex combination of these points only. Figure~\ref{fig:fcfw} plots the distance of the iterate $x_t$ to the target $x^*$ in the $\ell_p$-norm as a function of its cardinality, as given by the construction of the algorithm, for three arbitrary values $p=2$, $p=3$, and $p=7$.

\begin{figure}[H]
 \centering{\includegraphics[scale=0.6]{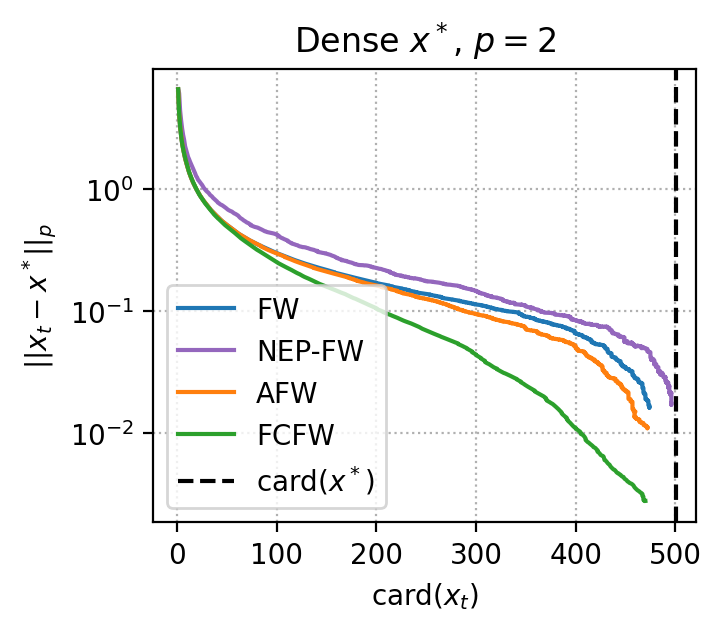}
 \hspace{10mm}\includegraphics[scale=0.6]{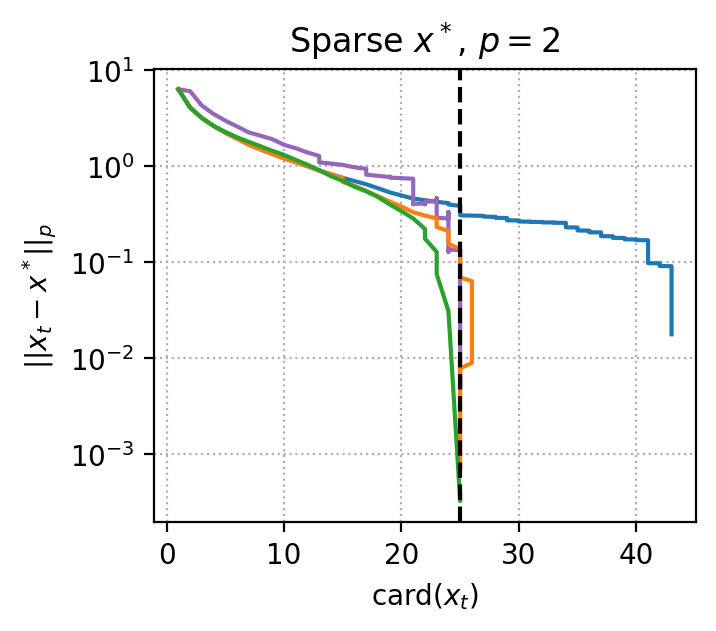}
 
 \vspace{2mm}\includegraphics[scale=0.6]{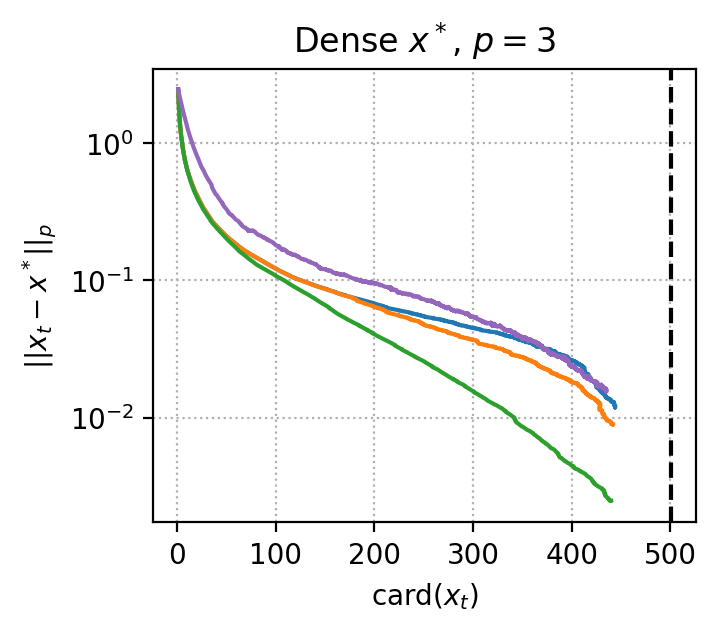}
 \hspace{10mm}\includegraphics[scale=0.6]{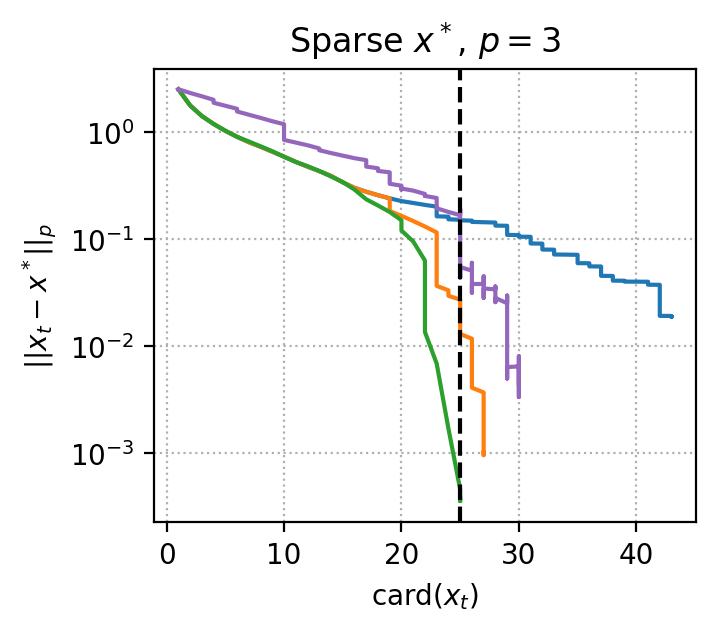}
 
 \vspace{2mm}
 \includegraphics[scale=0.6]{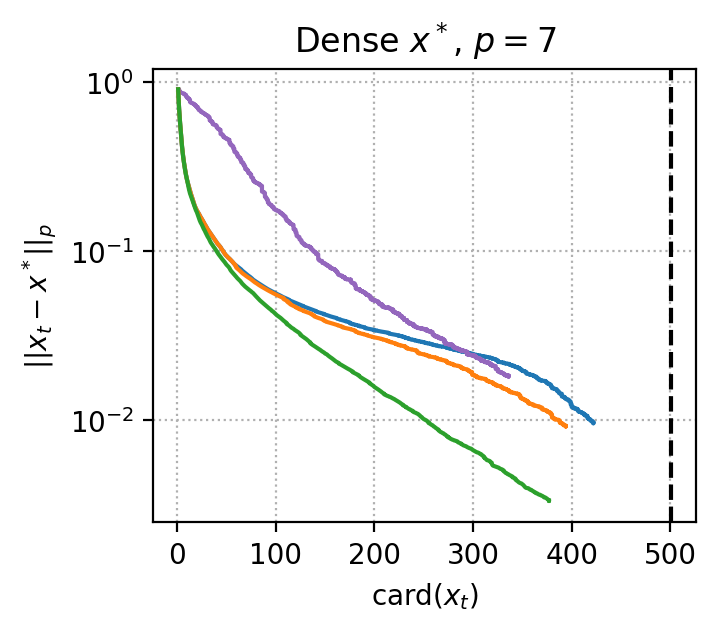}
 \hspace{10mm}\includegraphics[scale=0.6]{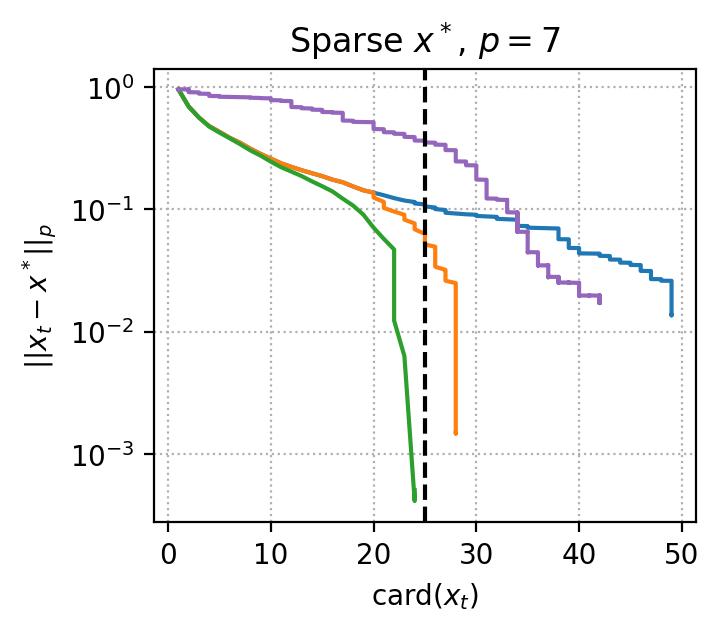}}
 \caption{Accuracy vs.~cardinality of the iterates generated by FW, NEP-FW, AFW, and FCFW.}
 \label{fig:fcfw}
\end{figure}

As expected, FCFW is the best performing algorithm, followed by AFW. In the instance ``Sparse $x^*$, $p=2$'', we can see that AFW took a full away step, which decreased the cardinality of the iterate by $1$. In the limit, NEP-FW improves on FW particularly when the target $x^*$ is sparse; note that we chose $x_0\in\mathcal{V}$ randomly and did not use a warm start. Only FCFW is able to recover an optimal convex decomposition on each instance, i.e., to find a solution with arbitrary accuracy and cardinality no greater than $\operatorname{card}(x^*)$. Table~\ref{tab:card} reports the cardinality of the solution obtained by each algorithm to reach accuracy $\epsilon=0.02$.

\begin{table}[h]
 \caption{Cardinality of the first iterate $x_t$ satisfying $\|x_t-x^*\|_p<0.02$.}
\label{tab:card}
\centering{
 \begin{tabular}{rrrrrrr}
  \toprule
  $\operatorname{card}(x^*)$&$p$&\textbf{FW}&\textbf{NEP-FW}&\textbf{AFW}&\textbf{FCFW}\\
  \midrule
  $501$&$2$&$470$&$496$&$453$&$368$\\
  $25$&$2$&$43$&$25$&$26$&$25$\\
  $501$&$3$&$417$&$413$&$388$&$275$\\
  $25$&$3$&$42$&$29$&$25$&$22$\\
  $501$&$7$&$349$&$325$&$294$&$175$\\
  $25$&$7$&$49$&$40$&$28$&$22$\\
  \bottomrule
 \end{tabular}
 }
\end{table}

\subsection{Lower bound}
\label{sec:lower}

We compare FW, AFW, and FCFW to the lower bound in \cite[Sec.~5.1]{mirrokni17cara}. Let $\mathcal{C}$ be the convex hull of the $\ell_p$-normalized columns of the Hadamard matrix $H_n$ of dimension $n$ from Sylvester's construction, i.e., $\mathcal{C}=\operatorname{conv}(H_n/n^{1/p})$, and let $x^*=(H_n/n^{1/p})1/n=e_1/n^{1/p}$ be the uniform convex combination of the columns, where $e_1\in\mathbb{R}^n$ denotes the first canonical vector. In this setting, \cite[Thm.~5.3]{mirrokni17cara} claims that for any $x\in\operatorname{conv}(H_n/n^{1/p})$ satisfying $\|x-x^*\|_p\leq\epsilon$, then $x$ has cardinality $s\geq\min\{1/\epsilon^2,n\}$. However, in their proof they use the inequality
\begin{align*}
 \frac{1}{\epsilon^2+1/n}
 \geq\frac{1}{\max\{\epsilon^2,1/n\}},
\end{align*}
which does not hold. Hence, for completeness, we present a minor correction to their lower bound in Theorem~\ref{th:lower}.

\begin{theorem}
\label{th:lower}
 Let $p\in\left[2,+\infty\right[$, $n\in\{2^k\mid k\in\mathbb{N}\}$, $H_n$ be the Hadamard matrix of dimension $n$ from Sylvester's construction, $\mathcal{C}=\operatorname{conv}(H_n/n^{1/p})$ be the convex hull of the $\ell_p$-normalized columns of $H_n$, and $x^*= e_1/n^{1/p}\in\mathcal{C}$. Let $\epsilon>0$ and $x\in\mathcal{C}$ such that $\|x-x^*\|_p\leq\epsilon$. Then $x$ is the convex combination of at least $1/(\epsilon^2+1/n)$ vertices.
\end{theorem}

Figure~\ref{fig:lower} compares FW, AFW, FCFW, and the corrected lower bound $s\in\llbracket1,n\rrbracket\mapsto\epsilon=\sqrt{1/s-1/n}$, for $n=64$ and two arbitrary values $p=4$ and $p=13$.

\begin{figure}[h]
 \centering{\includegraphics[scale=0.6]{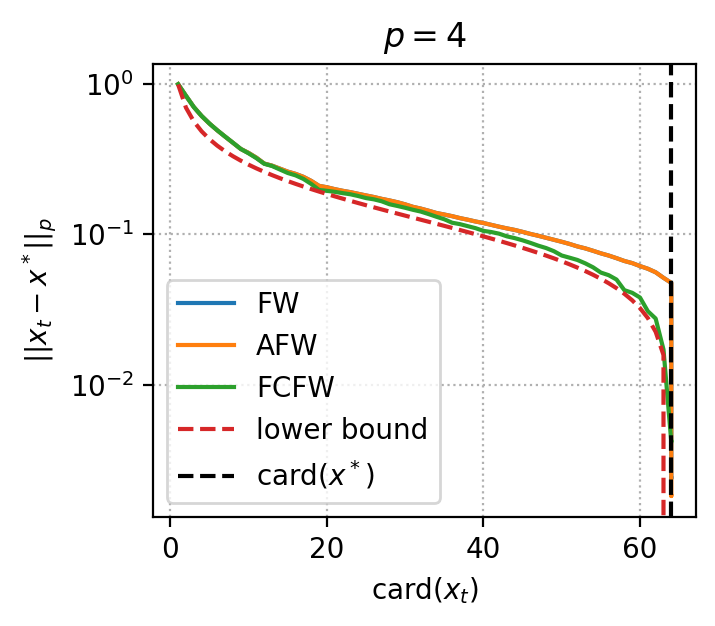}
 \hspace{10mm}\includegraphics[scale=0.6]{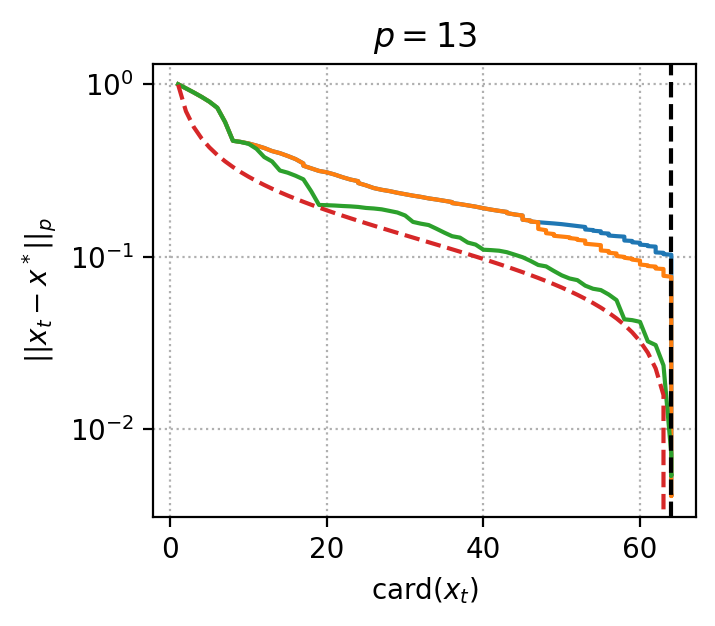}}
 \caption{Cardinality of the iterates produced by FW, AFW, and FCFW, and the lower bound from Theorem~\ref{th:lower}.}
 \label{fig:lower}
\end{figure}

FCFW again demonstrates its performance and almost matches the lower bound. This highlights its significance for the approximate Carath\'eodory problem. However, it remains an open problem to derive a precise convergence rate for FCFW, as the current analysis is transferred from the analyses of FW and AFW \cite{lacoste15}.

\section{Final remarks}
\label{sec:ccl}

We have demonstrated that the Frank-Wolfe algorithm provides a simple implementation of a solution with cardinality $\mathcal{O}(pD_p^2/\epsilon^2)$ to the approximate Carath\'eodory problem in the $\ell_p$-norm, where $p\in\left[2,+\infty\right[$. When $x^*$ is in the interior of $\mathcal{C}$, which may be likely in practice, the algorithm naturally adapts and generates a solution with cardinality $\mathcal{O}(p(D_p/r_p)^2\ln(1/\epsilon))$, where $r_p$ is the radius of an ball centered at $x^*$ and included in $\mathcal{C}$. This is in contrast with the method in \cite{mirrokni17cara} which requires knowledge of $r_p$. The Frank-Wolfe algorithm also adapts to the geometry of $\mathcal{C}$, and generates a solution with an improved cardinality bound when $\mathcal{C}$ is uniformly convex. When $x^*$ is the convex combination of a subset of vertices with small diameter, better cardinality bounds are obtained via a variant of the Frank-Wolfe algorithm with an enhanced oracle. When $p\in\left[1,2\right[\cup\{+\infty\}$, new bounds are proposed via a nonsmooth variant of the algorithm. Lastly, we addressed the problem of finding sparse approximate projections in the $\ell_p$-norm. In practice, when $p\in\left[2,+\infty\right[$, the Fully-Corrective Frank-Wolfe algorithm is very efficient and can generate a solution with near-optimal cardinality. However, a precise estimation has yet to be derived.

\subsection*{Acknowledgments}
\label{sec:ack}

Research reported in this paper was partially supported by NSF CAREER Award CMMI-1452463 and the Deutsche Forschungsgemeinschaft (DFG) through the DFG Cluster of Excellence MATH+. We thank Grigory Ivanov for providing the cardinality bound from \cite[Lem.~D]{bourgain89} in the case $p\in\left]1,2\right[$ and for bringing up interesting pointers to related problems.

\appendix

\section{Additional proofs}
\label{apx:proofs}

\begin{lemma}[Lemma~\ref{lem:init}]
 Let Assumption~\ref{aspt} hold and consider FW (Algorithm~\ref{fw}) with the open-loop strategy~\eqref{open} or the closed-loop strategy~\eqref{closed}. Then
 \begin{align*}
  f(x_1)-\min_\mathcal{C}f
  \leq\frac{LD^2}{2}.
 \end{align*}
\end{lemma}

\begin{proof}
 Let $x^*\in\argmin_\mathcal{C}f$ be a solution. If the open-loop strategy is used, then $\gamma_0=1$ so by smoothness of $f$, optimality of $v_0$ (Line~\ref{fw:lmo}), and convexity of $f$,
 \begin{align*}
  f(x_1)
   &\leq f(x_0)+\gamma_0\langle v_0-x_0,\nabla f(x_0)\rangle+\frac{L}{2}\gamma_0^2\|v_0-x_0\|^2\\
   &=f(x_0)+\langle v_0-x_0,\nabla f(x_0)\rangle+\frac{L}{2}\|v_0-x_0\|^2\\
   &\leq f(x_0)+\langle x^*-x_0,\nabla f(x_0)\rangle+\frac{LD^2}{2}\\
   &\leq f(x^*)+\frac{LD^2}{2}.
 \end{align*}
 If the closed-loop strategy is used, then by smoothness of $f$ and by optimality of the strategy~\eqref{closed:opt},
 \begin{align*}
   f(x_1)
   &\leq f(x_0)+\gamma_0\langle v_0-x_0,\nabla f(x_0)\rangle+\frac{L}{2}\gamma_0^2\|v_0-x_0\|^2\\
   &=\min_{\gamma\in\left[0,1\right]}f(x_0)+\gamma\langle v_0-x_0,\nabla f(x_0)\rangle+\frac{L}{2}\gamma^2\|v_0-x_0\|^2\\
   &\leq f(x_0)+\langle v_0-x_0,\nabla f(x_0)\rangle+\frac{L}{2}\|v_0-x_0\|^2,
 \end{align*}
 and we conclude as before.
\end{proof}

\begin{theorem}[Theorem~\ref{th:hcgs}]
 Let Assumption~\ref{aspt2} hold and consider HCGS (Algorithm~\ref{hcgs}) with $\beta_t\leftarrow2(D_2/G_2)/\sqrt{t+2}$ for all $t\in\mathbb{N}$ and the open-loop strategy~\eqref{open}. Then for all $t\geq2$,
 \begin{align*}
  f(x_t)-\min_\mathcal{C}f
  \leq\frac{4G_2D_2}{\sqrt{t+1}}.
 \end{align*}
\end{theorem}

\begin{proof}
Let $x^*\in\argmin_\mathcal{C}f$ and $t\in\mathbb{N}$. By Lemma~\ref{lem:fbeta}, $f_{\beta_t}$ is convex and $1/\beta_t$-smooth with respect to the $\ell_2$-norm. By optimality of $v_t$ (Line~\ref{hcgs:lmo}), we have
 \begin{align*}
  f_{\beta_t}(x_{t+1})
  &\leq f_{\beta_t}(x_t)+\gamma_t\langle v_t-x_t,\nabla f_{\beta_t}(x_t)\rangle+\frac{\gamma_t^2}{2\beta_t}\|v_t-x_t\|_2^2\\
  &\leq f_{\beta_t}(x_t)+\gamma_t\langle x^*-x_t,\nabla f_{\beta_t}(x_t)\rangle+\frac{\gamma_t^2}{2\beta_t}\|v_t-x_t\|_2^2\\
  &\leq f_{\beta_t}(x_t)+\gamma_t(f_{\beta_t}(x^*)-f_{\beta_t}(x_t))+\frac{\gamma_t^2}{2\beta_t}\|v_t-x_t\|_2^2.
 \end{align*}
 By \cite[Lem.~4.2]{argyriou14}, $f_{\beta_t}(x^*)\leq f(x^*)$ and $f_{\beta_{t+1}}(x_{t+1})\leq f_{\beta_t}(x_{t+1})+(\beta_t-\beta_{t+1})G_2^2/2$, so
 \begin{align*}
  f_{\beta_{t+1}}(x_{t+1})-f(x^*)
  &\leq(1-\gamma_t)(f_{\beta_t}(x_t)-f(x^*))+\frac{\gamma_t^2D_2^2}{2\beta_t}+\frac{(\beta_t-\beta_{t+1})G_2^2}{2}\\
  &=\frac{t}{t+2}(f_{\beta_t}(x_t)-f(x^*))+G_2D_2\frac{\sqrt{t+2}}{(t+2)^2}+G_2D_2\frac{\sqrt{t+3}-\sqrt{t+2}}{\sqrt{(t+3)(t+2)}}.
 \end{align*}
 Let $\epsilon_s=f_{\beta_s}(x_s)-f(x^*)$ for all $s\in\mathbb{N}$. Then
 \begin{align*}
  (t+1)(t+2)\epsilon_{t+1}
  &=t(t+1)\epsilon_t+G_2D_2\frac{t+1}{\sqrt{t+2}}+G_2D_2\frac{(t+1)\sqrt{t+2}}{\sqrt{t+3}}(\sqrt{t+3}-\sqrt{t+2})\\
  &\leq t(t+1)\epsilon_t+G_2D_2\sqrt{t+2}+G_2D_2\frac{(t+1)\sqrt{t+2}}{\sqrt{t+3}}(\sqrt{t+3}-\sqrt{t+2}).
 \end{align*}
 By telescoping for $s\in\llbracket0,t\rrbracket$, we obtain
 \begin{align*}
  (t+1)(t+2)\epsilon_{t+1}
  &=G_2D_2\sum_{s=0}^t\sqrt{s+2}+G_2D_2\sum_{s=0}^t\frac{(s+1)\sqrt{s+2}}{\sqrt{s+3}}(\sqrt{s+3}-\sqrt{s+2})\\
  &\leq G_2D_2\sum_{s=0}^t\int_s^{s+1}\sqrt{u+2}\,\mathrm{d}u+G_2D_2\sum_{s=0}^t\frac{(s+1)\sqrt{s+2}}{\sqrt{s+3}}\sqrt{s+3}\\
  &\quad-G_2D_2\sum_{s=0}^{t}\frac{s\sqrt{s+1}}{\sqrt{s+2}}\sqrt{s+2}\\
  &=G_2D_2\left[\frac{2}{3}(u+2)^{3/2}\right]_0^{t+1}+G_2D_2\sum_{s=0}^t\frac{(s+1)\sqrt{s+2}}{\sqrt{s+3}}\sqrt{s+3}\\
  &\quad-G_2D_2\sum_{s=-1}^{t-1}\frac{(s+1)\sqrt{s+2}}{\sqrt{s+3}}\sqrt{s+3}\\
  &\leq\frac{2}{3}G_2D_2(t+3)^{3/2}+G_2D_2(t+1)\sqrt{t+2}.
 \end{align*}
 Thus,
 \begin{align*}
  f_{\beta_{t+1}}(x_{t+1})-f(x^*)
  &\leq\frac{2}{3}G_2D_2\frac{(t+3)^{3/2}}{(t+1)(t+2)}+\frac{G_2D_2}{\sqrt{t+2}}\\
  &\leq\frac{3G_2D_2}{\sqrt{t+2}},
 \end{align*}
 where the last inequality holds for all $t\geq1$. By \cite[Lem.~4.2]{argyriou14}, $f(x_{t+1})\leq f_{\beta_{t+1}}(x_{t+1})+\beta_{t+1}G_2^2/2$. Therefore, for all $t\geq1$,
 \begin{align*}
  f(x_{t+1})-f(x^*)
  &\leq\frac{3G_2D_2}{\sqrt{t+2}}+\frac{G_2D_2}{\sqrt{t+3}}\\
  &\leq\frac{4G_2D_2}{\sqrt{t+2}}.
 \end{align*}
\end{proof}

\bibliographystyle{abbrv}
{\small\bibliography{biblio}}

\end{document}